\documentclass{article}%
\usepackage{graphicx}
\usepackage{amsmath}
\usepackage{amsfonts}
\usepackage{amssymb}%
\providecommand{\U}[1]{\protect \rule{.1in}{.1in}}
\newtheorem{theorem}{Theorem}

\newtheorem{definition}[theorem]{Definition}

\newtheorem{lemma}[theorem]{Lemma}

\newtheorem{proposition}[theorem]{Proposition}

\newenvironment{proof}[1][Proof]{\textbf{#1.} }{\  \rule{0.5em}{0.5em}}
\begin{document}

\title{\textbf{Essential closed surfaces and finite coverings of negatively curved
cusped 3-manifolds }}
\author{Charalampos Charitos}
\maketitle

\begin{abstract}
The existence of essential clsoed surfaces is proven for finite coverings of
3-manifolds that are triangulated by finitely many topological ideal
tetrahedra and admit a regular, negatively curved, ideal structure.\bigskip

$\boldsymbol{Keywords}$ Incompressible surfaces - Ideal tetrahedra -
$CAT(0)$-spaces\smallskip

$\boldsymbol{Mathematics}$ $\boldsymbol{Subject}$ $\boldsymbol{Classification}%
$ 57M50 - 57M10

\end{abstract}

\section{ Introduction}

In his pioneering work \cite{Thurston}, Thurston constructed complete
hyperbolic structures of constant negative curvature equal to $-1$ on the
complement of certain knots and links of $\mathbb{S}^{3},$ by realizing them
as a finite union of regular, hyperbolic ideal tetrahedra. Subsequently, it
was proven that the interior of compact hyperbolic 3-manifolds with geodesic
boundary can be triangulated by hyperbolic ideal polyhedra and in some cases,
by hyperbolic ideal tetrahedra \cite{Fujii}, \cite{Kojima}. Inspired by these
works, as well as from other relative works (see for example \cite{Frigerio},
\cite{Frigerio-Petronio} and the references in them), in the present paper a
special class of non-compact, triangulated 3-manifolds $M$ is defined. The
manifolds $M$ are obtained by gluing along their faces a finite family of
topological ideal tetrahedra $(D_{i})_{i\in I},$ where each $D_{i}$ is
homeomorphic to a 3-simplex in $\mathbb{R}^{3}$ with its vertices removed. On
$M$ a unique metric $d:M\times M\rightarrow \mathbb{R}$ is defined such that,
each tetrahedron $D_{i}$ equipped with the induced metric is isometric to the
regular, hyperbolic ideal tetrahedron of $\mathbb{H}^{3}.$ The metric $d$ will
be referred to as \textit{regular, ideal structure} of $M$ while $M$ equipped
with such a structure $d$ will be called \textit{regular, cusped 3-manifold}.
Furthermore, if the sum of the dihedral angles around each edge of a regular,
cusped manifold $M$ is $\geq2\pi$ the structure $d$ will be referred to as
\textit{regular, negatively curved, ideal structure. }A \textit{linking
surface} is a closed normal surface contained in a neighborhood of some cusp
of $M,$ see Definition \ref{linking surface} below. \textit{ }

Considering 3-manifolds $M$ which admit an ideal triangulation, that is, a
triangulation by topological ideal tetrahedra, our goal is to find in $M$
essential closed surfaces, i.e. orientable incompressible surfaces which are
not parallel to a linking surface of $M.$ This will be done using in an
essential way the geometry of negatively curved, ideal structures. Let us
recall here that Thurston, who had constructed a hyperbolic structure on the
figure-8 knot complement $M_{8}$ by considering a triangulation by regular
ideal hyperbolic tetrahedra on it, had also proved in Theorem 4.11 of his
notes \cite{Thurston's Notes} that all incompressible surfaces on $M_{8}$ are
linking surfaces. Therefore, in order to find essential closed surfaces in a
cusped 3-manifold, or more generally, in a closed 3-manifold $M$ with infinite
fundamental group, we are obliged to pass to a finite covering of $M.$

Recall also that if $M$ is a compact, connected irreducible 3-manifold with
non-empty incompressible boundary and if $M$ is not covered by a product
$F\times I,$ where $F$ is a closed orientable surface, then there exists a
finite covering of $M$ containing an essential closed surface i.e. a closed,
incompressible non-boundary parallel surface \cite{CLR}. This result, as well
as our main theorem in the present work, can be considered as seminal results
of the famous virtual Haken conjecture which is attributed to Waldhausen
\cite{Wald} and which was proven by I. Agol \cite{Agol} using significant
works of many mathematicians. Corollary 1.1 of \cite{CLR} is a more general
result than Theorem \ref{main} below. However, the proofs of \cite{CLR} are
based on Thurston's hyperbolization theorem while the proof of our result is
direct, constructive and rather elementary.

Now, we assume that $M$ is triangulated by finitely many topological ideal
tetrahedra and obviously any finite covering space $\widetilde{M}$ of $M$ is
naturally equipped with such an ideal triangulation. Hence the normal surfaces
in $M$ or in $\widetilde{M}$ considered below refer to these ideal triangulations.

The main theorem of this paper is the following.

\begin{theorem}
\label{main}Let $M$ be an orientable 3-manifold triangulated by finitely many
topological ideal tetrahedra. If $M$ admits a regular, negatively curved,
ideal structure\textit{, then }there exists a finite covering space
$\widetilde{M}$ of $M$ containing an essential closed surface.
\end{theorem}

The proof of this theorem is obtained by combining two basic results. The
first one says that there exists a finite covering space $\widetilde{M}$ of
$M$ containing a closed normal surface which is not isotopic to a linking
surface, see Theorem \ref{existence normal surface}. The second one says that
if $M$ admits a regular, negatively curved, ideal structure then any normal
closed surface in $M$ is essential, see Theorem \ref{incompressible surface}.

\section{Definitions and Preliminaries}

Let $\overline{D}$ be a \textit{topological tetrahedron}, that is, a
topological space homeomorphic to the standard 3-simplex $\Delta^{3}$ in
$\mathbb{R}^{4},$ via a homeomorphism $f:\overline{D}\rightarrow \Delta^{3}.$
The images of the vertices, edges and faces of $\Delta^{3},$ under $f^{-1},$
will be called vertices, edges and faces respectively of $\overline{D}.$ A
\textit{topological ideal tetrahedron} $D$ is a topological tetrahedron
$\overline{D}$ with its vertices removed. Thus, the edges and the faces of $D$
are the edges and the faces of $\overline{D}$ without its vertices.

\begin{definition}
\label{def triangulation}A triangulated ideal 3-manifold $M$ is a non-compact,
orientable, topological 3-manifold equipped with two finite sets $\mathcal{D}$
and $\mathcal{F}$ such that:

(1) Each element $D_{i}\in \mathcal{D}$ is a topological ideal tetrahedron.

(2) Each element $f\in \mathcal{F}$ is a homeomorphism $f:A\rightarrow B,$
where $A$ and $B$ are distinct faces of disjoint tetrahedra $D_{i},$ $D_{j}$
belonging to $\mathcal{D}$ and for each face $A$ of a tetrahedron $D_{i}%
\in \mathcal{D}$, there exists an $f\in \mathcal{F}$ and a face $B$ of some
tetrahedron $D_{j}\in \mathcal{D},$ with $f$ sending $A$ to $B.$ The elements
of $\mathcal{F}$ are called the \textit{gluing maps} of $M.$

(3) As a topological space, $M$ is the quotient space of the disjoint union of
all tetrahedra $(D_{i})_{i\in I}$ in $\mathcal{D}$ by the equivalence relation
which identifies two faces $A,$ $B$ of tetrahedra $D_{i},$ $D_{j}$
respectively of $\mathcal{D}$, whenever these faces are related by a map
$f:A\rightarrow B$ belonging to the collection $\mathcal{F}$.

The subdivision of $M$ into tetrahedra of $\mathcal{D}$ will be called a
topological ideal triangulation of $M$ and will be also denoted by
$\mathcal{D}.$ An edge (resp. face) of some $D_{i}\in \mathcal{D}$ will be
called an edge (resp. face) of the triangulation $\mathcal{D}$. The deleted
vertices of the tetrahedra $D_{i}\in \mathcal{D}$ will be called ideal vertices
of $M.$

Thus, the 3-simplexes of $\mathcal{D}$ are the tetrahedra $D_{i},$ the
2-simplexes of $\mathcal{D}$ are their faces while the 1-simplexes of
$\mathcal{D}$ are the edges of $D_{i}$ which are also called edges of
$\mathcal{D}$. Let us denote by $\mathcal{D}^{(i)},$ $i=1,2,3$ the
$i$-skeleton of $\mathcal{D}$.

An edge $e$ of $\mathcal{D}$ is said to have index $k,$ $k\geq2,$ if for each
point $x\in e$ there exists a closed neighborhood of $x$ in $\mathcal{D}%
^{(2)}$ which is homeomorphic to $k$ closed half discs glued together along
their diameter.

The index of $e$ with respect to the triangulation $\mathcal{D}$ will be
denoted by $i_{\mathcal{D}}(e).$
\end{definition}

Subsequently, we assume that the topological ideal triangulation $\mathcal{D}$
is fixed. A length metric $d$ can be defined on $M$ as follows:

Let $T_{0}$ be a regular hyperbolic ideal tetrahedron of the hyperbolic space
$\mathbb{H}^{3}.$ $T_{0}$ has all his dihedral angles equal to $\pi/3.$ As
each tetrahedron $D_{i}$ of $\mathcal{D}$ is homeomorphic to $T_{0}$ via an
homeomorphism $h_{i},$ we may equip $D_{i}$ with a metric so that $h_{i}$ is
an isometry. Now, by assuming that the gluing maps of $\mathcal{F}$ are
isometries between hyperbolic ideal triangles a unique length metric $d$ can
be defined naturally on the whole manifold $M.$ The triangulation of $(M,d)$
by regular hyperbolic ideal tetrahedra will be also denoted by $\mathcal{D}$
and will be called a \textit{regular, ideal triangulation; }the structure $d$
defined on $M$ will be referred to as \textit{regular, ideal structure.
\smallskip}

\noindent \textbf{Notation} \textit{Henceforth, it will be assumed that }%
$M$\textit{ has a fixed topological ideal triangulation }$\mathcal{D}$\textit{
which gives rise to a unique regular, ideal structure }$d.$\textit{ The
manifold }$M$\textit{ equipped with the regular, ideal structure }$d$\textit{
will be denoted by }$M_{d}$\textit{.\smallskip}

Gluing the hyperbolic ideal tetrahedra $D_{i}$ in order to build $M_{d},$ the
ideal vertices of $D_{i}$ are separated into finite classes so that the ideal
vertices of each class match together and form the \textit{cusps} of $M_{d}%
.$\textit{ } Thus, $M_{d}$ equipped with the structure $d$ will be
called\textit{ regular, cusped 3-manifolds}.
\begin{figure}[ptb]
\begin{center}
\includegraphics[scale=0.6]{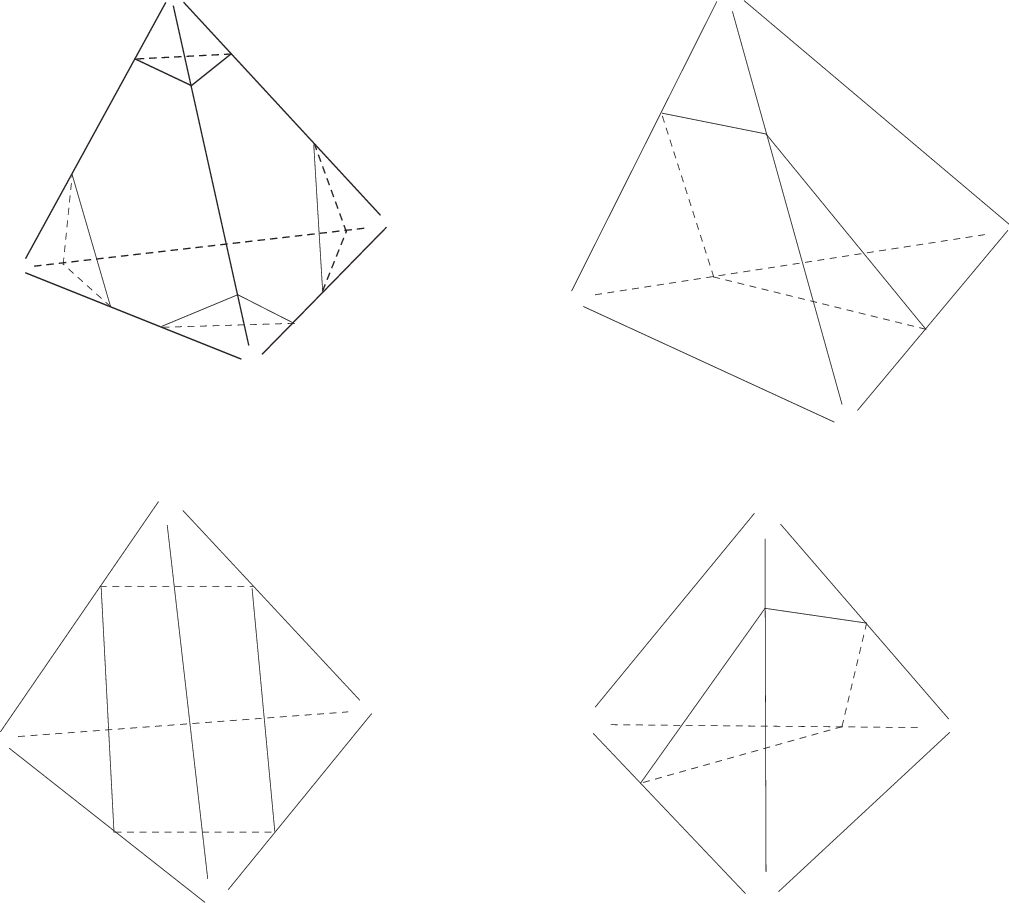}
\end{center}
\caption{The seven disc types.}%
\end{figure}
The manifold $M_{d}$, or generally a metric space, is called \textit{geodesic}
metric space if for any two points $p,q\in M_{d}$ there is a path, say
$[p,q],$ joining these points and whose length is equal to the distance
$d(p,q).$

The manifold $M_{d}$ has the following basic properties:

\begin{itemize}
\item $M_{d}$\textit{ is a complete, geodesic metric space}.
\end{itemize}

In fact, since all tetrahedra are regular, for each cusp $v$ of $M_{d},$ all
horospherical sections in a neighborhood of $v$ of each tetrahedron $D_{i}$
which has $v$ as an ideal vertex, fit together forming a closed surface
$S_{v}$ which is the geometrical link of $v$ in $M_{d}.$ Lemma 3.1 of
\cite{Char-Pap} or Theorem 3.4.23 of \cite{Thurston} implies that $M_{d}$ is a
complete space. On the other hand, the manifold $M_{d}$ is a locally compact,
complete length space and this implies that $M_{d}$ is a geodesic space (see
Hopf--Rinow Theorem in Proposition 3.7, p. 35 of \cite{Bridson-Haefliger}%
).\smallskip

Now, let $e$ be an edge of the regular ideal triangulation $\mathcal{D}$ of
$M_{d}.$ If $\theta(e,D_{i})$ is the dihedral angle of $D_{i}$ around $e$ then
$\theta(e,D_{i})=\pi/3.$ Let us denote by $\theta_{\mathcal{D}}(e)$ the sum of
all dihedral angles $\theta(e,D_{i})$ over all tetrahedra $D_{i}$ which share
$e$ as a common edge. Thus, $\theta_{\mathcal{D}}(e)\geq2\pi$ if and only if
$i_{\mathcal{D}}(e)\geq6.$

\begin{itemize}
\item $M_{d}$\textit{ has curvature less than or equal to }$-1$\textit{ i.e.
}$M_{d}$\textit{ satisfies locally the }$CAT(-1)$\textit{ inequality}
\textit{provided that }$i_{\mathcal{D}}(e)\geq6.$
\end{itemize}

In fact, since $\theta_{\mathcal{D}}(e)\geq2\pi$ for each edge $e$ of $M_{d}$,
it follows from \cite{Paulin}, Theorem 3.13, that $M_{d}$ has curvature less
than or equal to $-1.$ For this reason $M_{d}$ will be called a\textit{
regular, negatively curved cusped manifold}. The universal covering
$\widetilde{M_{d}}$ of $M_{d}$ satisfies globally the $CAT(-1)$ inequality, as
it follows from a theorem of Cartan--Hadamard--Aleksandrov--Gromov (see
Theorem 2.21 in \cite{Paulin}). Furthermore, the manifold $\widetilde{M_{d}}$
is homeomorphic to $\mathbb{R}^{3}$ and its visual boundary $\partial
\widetilde{M_{d}},$ which is defined via geodesic rays emanating from a base
point, is homeomorphic to the 2-sphere $S^{2}$ (see for example
\cite{Bridson-Haefliger}, Example (3), p. 266). \smallskip

For a regular, negatively curved, cusped manifold $M_{d},$ if $e$ is an edge
of $\mathcal{D}$ and if $\theta_{\mathcal{D}}(e)>2\pi$ then $e$ will be called
a \textit{singular edge} of $\mathcal{D}.$ If $\theta_{\mathcal{D}}(e)=2\pi$
for each $e\in \mathcal{D}^{(1)}$ then $M_{d}$ has everywhere constant
curvature $-1.$ In \cite{Thurston}, \cite{Fujii}, \cite{Kojima}, there are
examples of cusped 3-manifolds of constant curvature $-1$ which admit an ideal
triangulation. Allowing some edges $e$ to have $\theta_{\mathcal{D}}(e)>2\pi$
we may enrich the class of negatively curved cusped manifolds. For instance in
example 6.3 of \cite{Char-Pap}, such a manifold is constructed. Also
Proposition \ref{finite branched covering} below shows that the class of these
manifolds is really rich.
\begin{figure}[ptb]
\begin{center}
\includegraphics[scale=0.58]{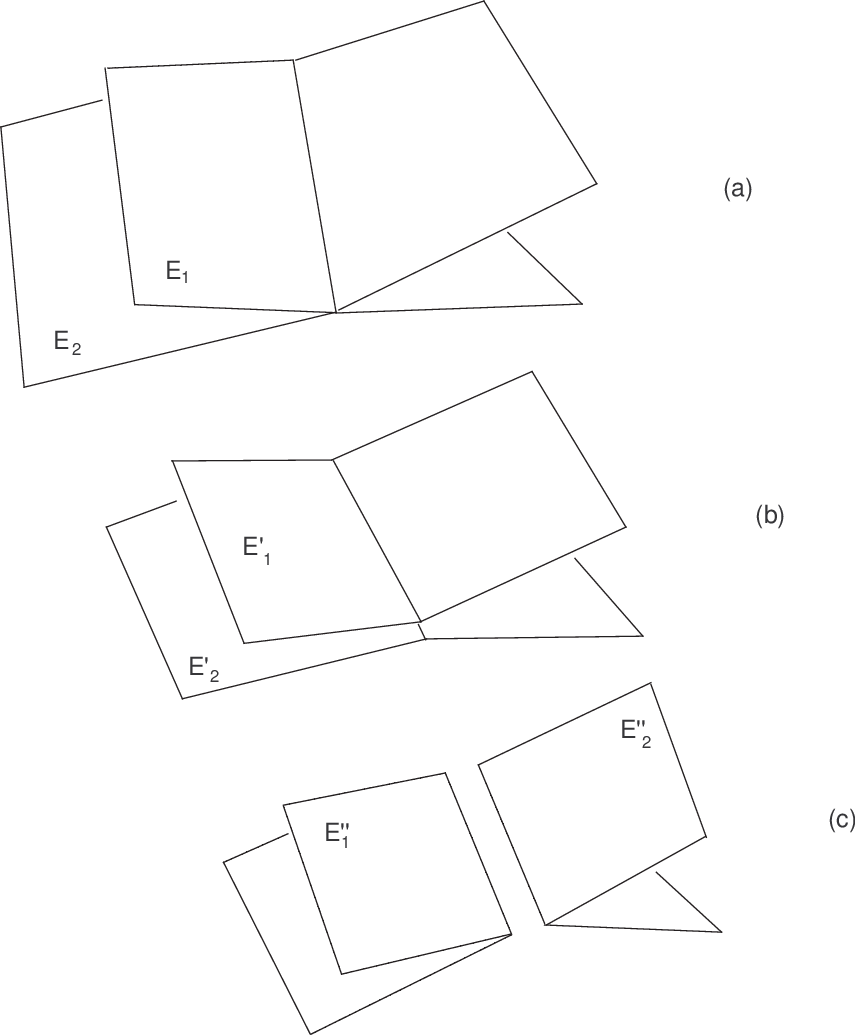}
\end{center}
\caption{Normal exchange of disc types.}%
\end{figure}
We recall now basic facts about normal surface theory which is mainly due to
Haken. A detailed exposition can be found in \cite{Jaco} or \cite{Hemion} and
similar definitions can be given for the triangulated ideal manifold $M.$
Usually, standard definitions are given either in the $PL$ or in the smooth
category since the sense of transversality is needed. In our context, and
without loss of generality we may assume that each tetrahedron of
$\mathcal{D}$ is hyperbolic and ideal and hence smooth. Thus, for the
existence of normal surfaces below, it is more convenient to work with $M_{d}$
which is triangulated by hyperbolic ideal tetrahedra.

There are seven kinds of special discs properly embedded in a tetrahedron
$D_{i}\in \mathcal{D}$ which are transverse to the faces of $D_{i}$ and which
are shown in Figure 1. These discs will be generally referred to as
\textit{disc types}. Thus, for each tetrahedron we have four triangular disc
types and three quadrilateral disc types.

Now, let $S$ be a closed surface equipped with a smooth structure. The term
\textit{smooth} for a map $f:S\rightarrow M_{d}$ will be used below in the
following sense: For each $x\in S$ with $f(x)\in \sigma,$ where $\sigma
\in \mathcal{D}^{(i)},$ $i=1,2,3,$ there exists a neighborhood $U_{\sigma}$ of
$f(x)$ in $\sigma$ such that the restriction $f_{|f^{-1}(U_{\sigma})\cap
S}:f^{-1}(U_{\sigma})\cap S\rightarrow \sigma$ is a smooth embedding.

A \textit{singular normal surface} in $M_{d}$ is a smooth map $f:S\rightarrow
M_{d}$ such that: $f$ is transverse to each simplex $\sigma \in \mathcal{D}%
^{(i)}$ and the intersection of $f(S)$ with each tetrahedron $\Delta_{i}$ of
$\mathcal{D}$ is a finite collection of discs types.

If $f$ is $1-1$ then the map $f:S\rightarrow M_{d}$ will be called
\textit{normal surface} and in this case $S$ will be identified with $f(S).$

\begin{definition}
\label{essential def}An essential closed surface in $M_{d}$ is an orientable
normal incompressible surface in $M_{d}$ which is not parallel to a linking surface.
\end{definition}

A well-known result, which is obviously valid for hyperbolic, ideal
triangulations, confirms that an incompressible surface $S$ in $M_{d}$ can be
isotoped to a normal surface with respect to $\mathcal{D}$ (see for example
Theorem 5.2.14 in \cite{Schultens}). The converse is not generally true for
compact triangulated 3-manifolds. However, from Theorem \ref{main} it results
that any normal closed surface in $M_{d}$ is essential provided that
$i_{\mathcal{D}}(e)\geq6$ for each edge $e$ of $\mathcal{D}.$

The following definitions are recalled from \cite{Jaco-Rubinstein}. Let $S$ be
a closed orientable surface. A \textit{normal homotopy} is defined to be a
smooth map $\xi:S\times \lbrack0,1]\rightarrow M_{d}$ so that for each fixed
$t\in \lbrack0,1]$ the surface $S_{t}$ given by $\xi_{|S\times \{t\}}$ is a
singular normal surface. $\xi$ is a \textit{normal isotopy} if in addition,
each $S_{t}$ is embedded. The \textit{normal homotopy class} $\mathcal{N}(f)$
of a normal or singular normal surface $f:S\rightarrow M_{d}$ is defined as
the set of all normal or singular normal surfaces $g:S\rightarrow M_{d}$ which
are normally homotopic to $f.$

\section{Properties of ideal triangulations}

In order to state our results we need some terminology.

\begin{definition}
\label{terminology triangulation} Let $X$ be a triangulated ideal 3-manifold
equipped with a topological ideal triangulation $\mathcal{X}.$

We will say that $\mathcal{X}$ has the \textit{property of the unique common
simplex } if for every two tetrahedra $D,$ $D^{\prime}$ of $\mathcal{X}$ one
of the following three cases can happen: (1) $D\cap D^{\prime}=\varnothing;$
(2) $D\cap D^{\prime}=e,$ where $e$ is an edge of $\mathcal{X};$ (3) $D\cap
D^{\prime}=F,$ where $F$ is an entire face of $\mathcal{X}.$

If $\partial X\neq \emptyset$ then each face of $\mathcal{X}$ which belongs to
two hyperbolic tetrahedra will be called \textit{interior face} of
$\mathcal{X}$, otherwise will be called \textit{exterior face} of
$\mathcal{X}$ or boundary \textit{face} of $X.$ Also, each edge of an exterior
face of $\mathcal{X}$ will be called \textit{exterior edge} of $\mathcal{X},$
otherwise it will be called \textit{interior edge} of $\mathcal{X}.$
\end{definition}

Following Definition \ref{def triangulation}, all the faces of the manifold
$M$ in Theorem \ref{main} are interior faces. Thus, $M$ will always be a
triangulated ideal 3-manifold equipped with a fixed topological ideal
triangulation $\mathcal{D}$ such that all faces of $\mathcal{D}$ are interior
faces i.e. $\partial M=\varnothing.\mathcal{\ }$ In Definition
\ref{def triangulation} we have also assumed that two disjoint faces of the
same tetrahedron of $\mathcal{D}$ are not glued together. Actually, this
assumption is redundant when we consider negatively curved, ideal structures
on $M.$

We will start this section with some topological results which concern the
covering and branched covering spaces of a triangulated ideal 3-manifold $M.$
In order to state our results we need some additional terminology.

First we need the following lemma.

\begin{lemma}
\label{special triangulation}There exists a triangulated ideal manifold $B$
equipped with a topological ideal triangulation $\mathcal{B}$ consisting of
the ideal tetrahedra $D_{i}$ of $\mathcal{D},$ such that:

(1) $B$ is homeomorphic to the closed unit 3-ball $D^{3}$ with finitely many
points removed from $\partial D^{3}.$

(2) $M$ is obtained from $B$ by gluing pairwise, via gluing maps of $M,$ the
boundary faces of $B.$

(3) All the edges of $\mathcal{B}$ are exteriors edges.

(4) $\mathcal{B}$ has the \textit{property of the unique common simplex.}
\end{lemma}

\begin{proof}
Let $D_{1}$ be an arbitrary ideal tetrahedron of $\mathcal{D}.$ There is a
tetrahedron $D_{2}$ of $\mathcal{D}$ which is glued to $D_{1},$ via a gluing
map of $M,$ and we set $B_{2}=D_{1}\cup D_{2}.$ To the topological ideal
polyhedron $B_{2}$ we glue a third ideal tetrahedron $D_{3}$ of $\mathcal{D}$,
disjoint from the previous ones $D_{1},D_{2},$ along a boundary face of
$B_{2}.$ Note that the gluing of $B_{2}$ with $D_{3},$ as well as, all the
gluings below, are always performed via the gluing maps of $M.$

By induction, to the triangulated ideal polyhedron $B_{k-1}=B_{k-2}\cup
D_{k-1},$ where $D_{k-1}$ is a tetrahedron of $\mathcal{D},$ we glue a new
tetrahedron $D_{k}$ of $\mathcal{D}$ which is disjoint from $D_{1}%
,..,D_{k-1}.$ From our construction it follows that each $B_{k}$ is
homeomorphic to the closed unit 3-ball $D^{3}$ with finitely many points
removed from $\partial D^{3}$ and that the ideal triangulation of $B_{k}$ has
the property of the unique common simplex. On the other hand, since $M$
consists of finitely many ideal tetrahedra, there is some $n$ such that
$B=B_{n}$ satisfies all properties (1)-(4) of our lemma.
\end{proof}

\begin{proposition}
\label{finite covering} There is a finite covering space $N$ of $M$ and a
topological ideal triangulation $\mathcal{E}$ of $N$ such that $\mathcal{E}$
has the property of the unique common simplex.
\end{proposition}

\begin{proof}
From Lemma \ref{special triangulation} we may cut and open $M$ along
appropriate faces of $\mathcal{D}$ and obtain a manifold $M_{1}$ with boundary
faces $A_{i}^{\prime},$ $A_{i}^{\prime \prime},$ $i=1,..,m,$ such that:

(1) $M_{1}$ is connected and if we denote by $\mathcal{D}_{1}$ the ideal
triangulation induced on $M_{1}$ by $\mathcal{D}$ then $\mathcal{D}_{1}$ has
the property of the unique common simplex.

(2) The faces $A_{i}^{\prime},$ $A_{i}^{\prime \prime}$ result by cutting $M$
along the face $A_{i}$ of $\mathcal{D}$ thus, when $A_{i}^{\prime}$ is glued
back to $A_{i}^{\prime \prime}$ for each $i$ (following the same gluing map) we
obtain $M.$

Obviously, the faces $A_{i}^{\prime},$ $A_{i}^{\prime \prime}$ belong to
different tetrahedra of $\mathcal{D}_{1}.$

Actually, we may assume that $M_{1}$ is homeomorphic to a 3-ball with finitely
many points removed from its boundary but this information does not matter here.

Now, we consider two copies $M_{1}^{(1)}$ and $M_{1}^{(2)}$ of $M_{1}$ and let
us denote by $A_{1}^{\prime(1)},$ $A_{1}^{\prime \prime(1)}$ (resp.
$A_{1}^{\prime(2)},$ $A_{1}^{\prime \prime(2)})$ the faces of $M_{1}^{(1)}$
(resp. of $M_{1}^{(2)}),$ corresponding to $A_{1}^{\prime}$ and $A_{1}%
^{\prime \prime}.$ Then we glue $M_{1}^{(1)}$ and $M_{1}^{(2)}$ by identifying
the pair $(A_{1}^{\prime(1)},A_{1}^{\prime \prime(1)})$ with the pair
$(A_{1}^{\prime \prime(2)},A_{1}^{\prime(2)}).$ This means that $A_{1}%
^{\prime(1)}$ is identified with $A_{1}^{\prime \prime(2)}$ and $A_{1}%
^{\prime \prime(1)}$ is identified with $A_{1}^{\prime(2)},$ via the gluing map
of $M$ which identifies $A_{1}^{\prime}$ and $A_{1}^{\prime \prime}.$ This rule
will be always applied when two pairs of faces are identified below.

\begin{figure}[ptb]
\begin{center}
\includegraphics[scale=0.7]{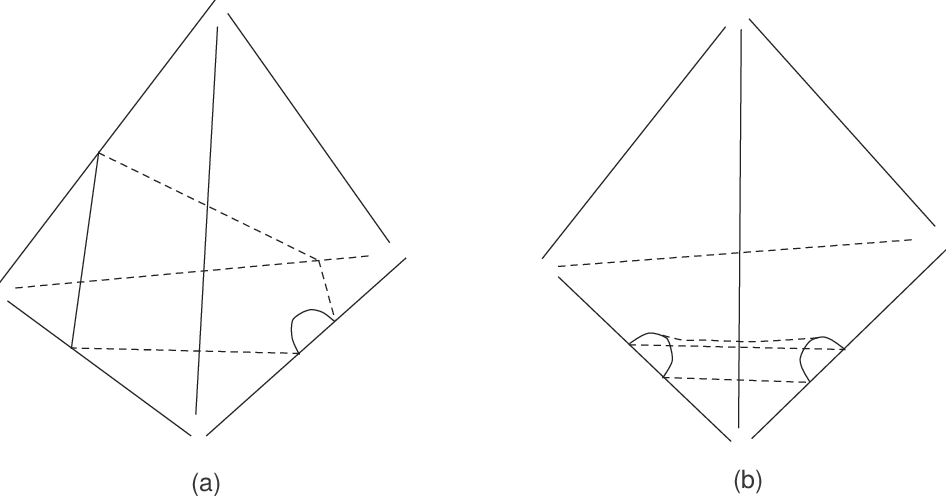}
\end{center}
\caption{Pseudo-triangular and tunnel disc types.}%
\end{figure}

Gluing $M_{1}^{(1)}$ and $M_{1}^{(2)}$ as before we obtain a connected
manifold $M_{2}$ with an ideal triangulation $\mathcal{D}_{2}.$ Since
$A_{1}^{\prime}$ and $A_{1}^{\prime \prime}$ belong to different tetrahedra of
$\mathcal{D}_{1}$ we may easily verify that:

(i) $\mathcal{D}_{2}$ has the property of the unique common simplex.

(ii) $M_{2}$ does not have boundary faces corresponding to the faces
$A_{1}^{\prime}$ and $A_{1}^{\prime \prime}$ of $M_{1}.$

On the contrary, $M_{2}$ has boundary faces $A_{2,k}^{\prime},$ $A_{2,k}%
^{\prime \prime},$ $k=1,..,n_{2}=2^{2-1},$ corresponding to the boundary faces
$A_{2}^{\prime}$ and $A_{2}^{\prime \prime}$ of $M_{1}.$ We consider again two
copies $M_{2}^{(1)}$ and $M_{2}^{(2)}$of $M_{2}$ and for $j=1,2,$ let
$A_{2,k}^{\prime(j)},$ $A_{2,k}^{\prime \prime(j)}$ be the faces of
$M_{2}^{(j)}$ corresponding to $A_{2,k}^{\prime},$ $A_{2,k}^{\prime \prime}.$
Our goal is to glue these faces in pairs and obtain a new manifold $M_{3}.$
For this, we glue $M_{2}^{(1)}$ with $M_{2}^{(2)}$ by identifying the pair of
faces $(A_{2,k}^{\prime(1)},A_{2,k}^{\prime \prime(1)})$ with $(A_{2,k}%
^{\prime \prime(2)},A_{2,k}^{\prime(2)}),$ $k=1,2,..,n_{2}.$

Remark that all the faces $A_{2,k}^{\prime}$ and $A_{2,k}^{\prime \prime},$
$k=1,..,n_{2}$ belong to different tetrahedra of $\mathcal{D}_{2}.$ Therefore,
exactly as before we get a connected manifold $M_{3}$ equipped with an ideal
triangulation $\mathcal{D}_{3}$ such that:

(i) $\mathcal{D}_{3}$ has the property of the unique common face.

(ii) $M_{3}$ does not have boundary faces corresponding to the boundary faces
$A_{1}^{\prime},$ $A_{1}^{\prime \prime}$ and $A_{2}^{\prime},$ $A_{2}%
^{\prime \prime}$ of $M_{1}.$

On the contrary, there are boundary faces $A_{3,k}^{\prime}$ and
$A_{3,k}^{\prime \prime},$ $k=1,..,n_{3}=2^{3-1}$ of $M_{3}$ corresponding to
the faces $A_{3}^{\prime}$ and $A_{3}^{\prime \prime}$ of $M_{1}.$

By repeating our procedure, we consider two copies $M_{3}^{(1)}$ with
$M_{3}^{(2)}$ of $M_{3}$ and gluing them appropriately we obtain a manifold
$M_{4}$ equipped with an ideal triangulation $\mathcal{D}_{4}.$ Then, we can
easily verify that $(M_{4},\mathcal{D}_{4})$ verify the analogous properties
(i), (ii) above and so on. Thus, after finitely many steps, we end with a
manifold $N$ without boundary faces. From our construction follows that a
finite covering space $p:N\rightarrow M$ is obtained. Obviously $N$ is
equipped with a topological ideal triangulation $\mathcal{E}$ which has the
property of the unique common simplex.
\end{proof}

Now we will prove two basic existence statements. The first concerns the
existence of normal surfaces and the second the existence of a branched
covering $M^{\prime}$ of $M$ equipped with an ideal triangulation
$\mathcal{D}^{\prime}$ such that: the branch locus consists of edges of
$\mathcal{D}$ and the index of each edge of $\mathcal{D}^{\prime}$ is $\geq6.$

In order to deal with the problem of existence of normal surfaces we may work,
without loss of generality, with the manifold $M_{d}$ instead of $M.$ Part of
the following definition is borrowed from \cite{Jaco-Rubinstein}.

\begin{definition}
\label{linking surface}To each cusp $c_{i}$ of $M_{d}$ corresponds a normal
surface $C_{i}$ consisting of triangular disc types. The surface $C_{i}$ is
contained in a neighborhood $U_{i}$ of $c_{i}$ in $M_{d}$ which is
homeomorphic to $C_{i}\times \lbrack0,\infty).$ The neighborhood $U_{i}$ will
be called trivial neighborhood of $c_{i}.$

Each normal surface in $U_{i},$ modulo normal isotopy, will be parallel to
$C_{i}$ and will be referred to as a \textit{linking surface}, following
\cite{Jaco-Rubinstein}. A singular normal surface lying in $U_{i}$ will be
called a \textit{multiple linking surface}.

A closed curve $a$ in $U_{i}$ which is non-contractible in $U_{i}$ will be
called an essential cuspidal curve. Obviously $a$ is freely homotopic with a
curve $a^{\prime}$ belonging to a linking surface $C_{i}.$
\end{definition}

\begin{lemma}
\label{horospheres} If $M_{d}$ is a \textit{regular, negatively curved cusped
manifold then }each linking surface $C$ in a trivial neighborhood $U$ of a
cusp $c$ of $M_{d}$ is incompressible.
\end{lemma}

\begin{proof}
The cusp $c$ is an ideal vertex of some hyperbolic ideal tetrahedron $D$ of
the triangulation $\mathcal{D}$ of $M_{d}.$ Obviously there is a point
$\widetilde{c}$ in the boundary $\partial \widetilde{M_{d}}$ of $\widetilde
{M_{d}}$ which is an ideal vertex of some hyperbolic ideal tetrahedron
$\widetilde{D}\subset \widetilde{M_{d}},$ where $\widetilde{D}$ is a lifting of
$D\subset M_{d}.$ Since $\partial \widetilde{M_{d}}$ is homeomorphic to $S^{2}$
it is not hard to prove the the horosphere $H$ corresponding to $\widetilde
{c}$ is homeomorphic to $\mathbb{R}^{2}.$ In fact, the geodesic which joins
$\widetilde{c}$ with a point of $\partial \widetilde{M_{d}}-\{ \widetilde{c}\}$
intersects $H$ in a single point. Thus a bijection between $\partial
\widetilde{M_{d}}-\{ \widetilde{c}\}$ and $H$ is established which proves that
$H$ is homeomorphic to $\mathbb{R}^{2}.$ Finally, $H\subset \widetilde{M_{d}}$
is a lifting of $C$ and this implies that $C$ is incompressible in $M_{d}.$
\end{proof}

Now let $D_{i}$ be a tetrahedron of $\mathcal{D}$ and let two disjoint
quadrilateral disc types $E_{1},$ $E_{2}$ in $D_{i}.$ The disc types $E_{1},$
$E_{2}$ cross over each other transversely as in Figure 2(a), and we may
perform a surgery to construct disjoint triangular disc types within $D_{i}.$
Indeed, by cutting $E_{1}$ and $E_{2}$ along the line of intersection and
reassembling the pieces properly we take within $D_{i}$ the properly embedded
discs $E_{1}^{\prime},$ $E_{2}^{\prime}$ of Figure 2(b) or the properly
embedded discs $E_{1}^{\prime \prime},$ $E_{2}^{\prime \prime}$ of Figure 2(c).
After performing an isotopy, via properly embedded discs in $D_{i},$ the discs
$E_{1}^{\prime},$ $E_{2}^{\prime},$ $E_{1}^{\prime \prime},$ $E_{2}%
^{\prime \prime}$ are either triangular disc types or they have the form of
Figure 3(a). This latter disc will be abusively called a
\textit{pseudo-triangular disc type} of $D_{i}.$ The previous procedure, which
leads the intersected disc types $E_{1},$ $E_{2}$ to disjoint discs
$E_{1}^{\prime},$ $E_{2}^{\prime}$ (or $E_{1}^{\prime \prime},$ $E_{2}%
^{\prime \prime})$ will be referred to as \textit{surgery of quadrilateral disc
types within} $D_{i}.$

We assume now that the discs types $E_{1},$ $E_{2}$ are intersected
transversely in $D_{i}$ but that only one of them is triangular. Then, as
before, by cutting $E_{1},$ $E_{2}$ along the line of intersection and
reassembling the pieces properly we take within $D_{i}$ disc types either of
triangular or of quadrilateral form or, we take a disc having the form of
Figure 3(b) and which will be called disc of \textit{tunnel type} or abusively
a \textit{tunnel disc type}. The tunnel disc type is parallel to a disc
contained in $\partial D_{i}.$ This procedure, which leads the intersected
disc types $E_{1},$ $E_{2}$ to disjoint triangular disc types, quadrilateral
disc types or tunnel types will be referred to as \textit{surgery of
triangular disc types within} $D_{i}.$

Finally, we consider triangular discs types $E_{1},$ $E_{2}$ in $D_{i}.$ If
$E_{1},$ $E_{2}$ are intersected transversely, choosing one of them arbitrary,
we move it sufficiently close to an ideal vertex of $D_{i},$ via an isotopy by
parallel disc types. In this way we obtain disjoint triangular disc types.
This procedure, which leads the intersected triangular disc types $E_{1},$
$E_{2}$ to disjoint triangular disc types will be referred to as
\textit{trivial isotopy of disc types within} $D_{i}.$

\begin{definition}
A generalized normal surface in $M$ is an embedded surface $S$ in $M_{d}$ such
that for each hyperbolic ideal tetrahedron $D$ of $\mathcal{D}$ the
intersection $S\cap D$ can be, apart from the seven disc types defined in
Paragraph 2, a \textit{pseudo-triangular disc type or tunnel disc type.}

Let $S$ be generalized normal surface. We will say that $S$ contains a cyclic
tunnel if there is a subsurface $T$ of $S$ homeomorphic to $S^{1}\times
\lbrack0,1]$ such that $T$ is a finite union of tunnel disc types
$T_{1},..,T_{k}$ with $Int(T_{i})\cap Int(T_{j})=\emptyset$ for $i\neq j.$
\end{definition}

Now, let $D_{i},$ $i=1,...n$ be all the ideal tetrahedra of $\mathcal{D}$. As
we have seen, in each tetrahedron $D_{i}$ there are seven disc types. The
number of all considered disc types is $7n;$ we denote them by $E_{i}$ and to
each one we correspond a variable $x_{i}$, $i=1,..,a$ with $a=$ $7n.$

There are also $4n/2=2n$ 2-simpices $\sigma$ in $\mathcal{D}^{(2)}$ since by
hypothesis we have $n$ tetrahedra and each face belongs exactly to two
tetrahedra. In a 2-simplex $\sigma$ of $\mathcal{D}^{(2)}$ there are three
possible classes of arcs which belong to the boundary of disc types. Let
$\{l_{1},..,l_{m}\},$ $m=6n$ be these classes of arcs. Each $l_{j}$ has two
sides, looking into the one or into the other of the tetrahedra of $M_{d}$
which share the face $\sigma$ containing $l_{j}.$ In a purely abstract way, we
can go through and label the sides with the words \textquotedblleft
left\textquotedblright \ and \textquotedblleft right\textquotedblright.

For all numbers $i$ between $1$ and $a,$ and $j$ between $1$ and $m,$ we may
define the numbers $b_{ij}$ as follows:%

\[
b_{ij}=\left \{
\begin{array}
[c]{c}%
0\text{ if }l_{j}\text{ is not in }E_{i}\\
1\text{ if }E_{i}\text{ is on the left side of }l_{j}\\
-1\text{ if }E_{i}\text{ is on the right side of }l_{j}%
\end{array}
\right \}  .
\]

With this definition, the adjacency restriction can be formulated as a system
of linear equations:%

\begin{equation}%
{\displaystyle \sum \limits_{j=1}^{a}}
b_{ij}x_{i}=0\Leftrightarrow BX=0,\text{ where }B=[b_{ij}]. \tag{1}\label{1}%
\end{equation}

The problem is to prove that the linear system (\ref{1}) has positive
solutions, since to each positive solution of integer numbers corresponds a
singular normal surface $f:S\rightarrow M_{d}$ and furthermore that
$f:S\rightarrow M_{d}$ is not a multiple linking surface. Therefore the whole
problem is reduced to a classical problem of linear algebra which seeks for
conditions which guarantee the existence of positive solutions of a linear
system, see for example \cite{Hemion} and \cite{Dines}. Generally, this is not
an easy problem to deal with. For this reason we give below a geometric
construction of normal surfaces for triangulated ideal 3-manifolds.\medskip

Let $e$ be an arbitrary edge of $\mathcal{D}$ and let $Star(e)$ be the subset
of $M_{d}$ consisting of all tetrahedra of $\mathcal{D}$ having $e$ as a
common edge. Since $M_{d}$ has the property of the unique common simplex
$Star(e)$ is a particular simple set cosisting of $n$ hyperbolic ideal
tetrahedra $D_{1},..,D_{n}$ such that $e\subset D_{i}$ for each $i$ and
$D_{i}\cap D_{i+1(\operatorname{mod}n)}$ is a face of $\mathcal{D}.$ The form
of $Star(e)$ plays a key role in what follows.

Now, we are able to prove the following.

\begin{theorem}
\label{existence normal surface}We assume that the ideal triangulation
$\mathcal{D}$ of $M$ has the property of the unique common simplex. Then,

(a) There exists a singular normal surface $f:S\rightarrow M_{d}$ such that
$S$ is orientable and $f$ is an immersion.

(b) There exists an orientable normal surface $S_{0}$ in $M_{d}.$
\end{theorem}

\begin{proof}
From Lemma \ref{special triangulation} and using $M_{d}$ instead of $M,$ we
may find a manifold $B$ equipped with a hyperbolic ideal triangulation
$\mathcal{B}$ consisting of the hyperbolic ideal tetrahedra $D_{i}$ of
$\mathcal{D}$ and such that the properties (1)-(4) of Lemma
\ref{special triangulation} are satisfied. The only difference here, with
respect to Lemma \ref{special triangulation}, is that the gluing maps of
$M_{d}$ are isometries.

Let $E_{1}$ be a square disc type in a tetrahedron, say $D_{1}$ of $B.$ The
disc $E_{1}$ can be extended to a disc $C_{1}$ in $B$ with $\partial
C_{1}\subset \partial B.$ The extension of $E_{1}$ is done by adding
successively triangular disc types. Notice here that this extension of $E_{1}$
to the disc $C_{1}$ (by gluing only triangular disc types) is unique and hence
$C_{1}$ will be referred to as the extension of $E_{1}.$ If $c_{1}=\partial
C_{1}$ we remark that $c_{1}$ determines uniquely, up to isotopy, the disc
$C_{1}$ in $B.$ Also, $c_{1}$ intersects each boundary face at most in a
single arc. The latter follows easily since $c_{1}$ separates $\partial B.$

Similarly, starting from the other two quadrilateral disc types $E_{2},$
$E_{3}$ in $D_{1}$ we denote their extensions in $B$ by $C_{2},$ $C_{3}$
respectively and let $c_{2}=\partial C_{2},$ $c_{3}=\partial C_{3}.$

The curve $c_{1}$ consists of simple arcs, say $a_{1},..,a_{n},$ i.e.
$c_{1}=a_{1}\cup...\cup$ $a_{n}$ with $a_{i}\subset F_{i},$ where $F_{i}$ is a
boundary face of $B.$ Let $F_{i}^{\prime}$ be the face of $B$ which is glued
to $F_{i}$ to form $M_{d}.$ Assuming that $f_{i}$ is the gluing map between
$F_{i}$ and $F_{i}^{\prime}$ let $c_{1}^{\prime}=a_{1}^{\prime}\cup...\cup$
$a_{n}^{\prime}$ where $a_{i}^{\prime}=f(a_{i}).$ Generally $F_{1}^{\prime
}\cup...\cup$ $F_{n}^{\prime}$ is not connected and hence $c_{1}^{\prime}$ is
not connected too. Thus, we set $c_{1}^{\prime}=b_{1}^{\prime}\cup...\cup$
$b_{k}^{\prime}$ where each $b_{i}$ is a connected arc.

From the form of the set $Star(e)\subset M_{d},$ where $e\in \mathcal{D}%
^{(1)},$ we may find simple arcs $x_{i},$ $y_{i}$ arcs in $\partial B,$
$i=1,..,k-1$ such that:

(1) the curve $d_{1}^{\prime}=b_{1}^{\prime}\cup x_{1}\cup y_{1}\cup
b_{2}^{\prime}\cup...b_{k-1}^{\prime}\cup x_{k-1}\cup y_{-1}\cup$
$b_{k}^{\prime}$ is connected i.e. $d_{i}^{\prime}$ is a simple closed curve
in $\partial B;$

(2) the arcs $x_{i}$ and $y_{i}$ are identified in $M_{d},$ that is, the curve
$c_{1}^{\prime}=b_{1}^{\prime}\cup...\cup$ $b_{k}^{\prime}$ in $M_{d}$ is connected.

Now let $C_{1}^{\prime \prime}$ be the disc in $B$ bounded by $d_{1}^{\prime}.$
From the previous property (2) there exists a disc $C_{1}^{\prime}$ in $M_{d}$
with $\partial C_{1}^{\prime}=c_{1}^{\prime}.$ In fact, identyfing for each
$i$ the arcs $x_{i}$ and $y_{i}$ in $M_{d}$ the disc $C_{1}^{\prime}$ results
from $C_{1}^{\prime \prime}.$

Now $S_{1}=C_{1}\cup_{c_{1}\equiv c_{1}^{\prime}}C_{1}^{\prime}$ is the image
of a singular normal surface in $M_{d}.$

Below we need $C_{1}^{\prime}$ to be different from $C_{2}$ and $C_{3}.$
Actually, $C_{1}^{\prime}$ can coincide with $C_{2}$ or $C_{3}$ only if
$c_{1}^{\prime}$ is connected and coincides with $c_{2}$ or $c_{3}.$ In this
special case, if for example $c_{1}^{\prime}=c_{2},$ we work with $c_{3}$ in
the place of $c_{1}.$ Then the image $c_{3}^{\prime}$ of the curve $c_{3}$
obtained in the place of $c_{1}^{\prime}$ cannot coincide neither with $c_{1}$
nor $c_{2}.$ Therefore, we may always assume that $C_{1}^{\prime}\neq C_{2}$
and $C_{1}^{\prime}\neq C_{3}.$

Furthermore, $S_{1}$ is not a multiple linking surface since it contains the
disc type $E_{1}.$ Thus, we may construct a singular normal surface
$f:S\rightarrow M_{d},$ where $S$ is a closed surface with $f(S)=S_{1}.$

We will show that $S$ is orientable and that $f$ is an immersion. Indeed, from
the triangulation $\mathcal{D}$ and by means of $f,$ a cell decomposition
$\mathcal{T}$ is induced on $S.$ Each cell $T\in \mathcal{T}$ is either a
triangle or a quadrilateral in the sense that, a triangle (resp.
quadrilateral) $T\in \mathcal{T}$ is mapped by $f$ to a triangular (resp.
quadrilateral) disc type in some tetrahedron of $\mathcal{D}.$ As $f$ is $1-1$
on each cell $T\in \mathcal{T}$ we may assume that the images $f(T),$ intersect
each other in $M_{d}-$ $\mathcal{D}^{(1)}.$ From our construction $S_{1}$
induces at most two disc types in each tetrahedron of $\mathcal{D}.$
Therefore, all intersection points of $f(T),$ $T\in \mathcal{T}$ are
generically double points and since $f$ is assumed to be smooth, we deduce
that $f$ is an immersion. Finally, since $M_{d}$ is orientable all tetrahedra
$D_{i}$ can be oriented compatibly and thus all $T\in \mathcal{T}$ inherit
orientations from $D_{i}$ which are also compatible. Therefore $S$ is an
orientable surface.

Now we will derive the existence of a normal surface $S_{0}$ in $M_{d}$ from
the existence of $f:S\rightarrow M_{d}.$ For this, we perform within each
tetrahedron isotopies and surgeries following the next rules:

(I) \textit{By trivial isotopies of disc types in the appropriate ideal
tetrahedra, we may assume that any pair of triangular disc types }$E,$\textit{
}$E^{\prime}\subset$\textit{ }$f(S)$\textit{ which are contained in the same
}$D_{i}$\textit{ they do not intersect between them. }

Assumption (I) implies that a triangular disc type and a quadrilateral one
belonging to $f(S),$ may intersect in some tetrahedra. Recall also that, by
the construction of $f(S),$ on each tetrahedron of $\mathcal{D}$ they are
induced at most two disc types. These disc types are possibly of quadrilateral
type only within $D_{1}.$

(II) \textit{We fix a tetrahedron, say }$D_{i_{0}}$\textit{ in which a
triangular disc type }$P$\textit{ and a quadrilateral disc type }$Q$\textit{
appear. In }$D_{i_{0}}$\textit{ we perform a surgery of triangular disc type
between }$P$\textit{ and }$Q$\textit{ so that we obtain two disjoint disc
types of quadrilateral and triangular type, denoted by }$P_{i_{0}}$\textit{
and }$Q_{i_{0}}$\textit{ respectively. The surgeries performed to all the
others tetrahedra }$D_{i}$\textit{ are performed in a way compatible with the
surgery in }$D_{i_{0}}.$\textit{ This means that we take finitely many
components }$K_{1},$\textit{ }$K_{2},..,K_{n}$\textit{ which are orientable,
connected, closed surfaces and which are in general, generalized normal
surfaces.}

We remark here that the tetrahedron $D_{i_{0}}$ exists because of our
construction of $S_{1}.$ More precisely, $C_{1}$ has only one quadrilateral
disc type in $D_{1}$ and since we have assumed that $C_{1}^{\prime}$ is
different from $C_{2}$ and $C_{3}$ our remark follows.

We remark also that at most one pseudo-triangular disc type can appear after
performing the previous isotopy and surgery operations. Indeed, this follows
from the fact that only in $D_{1}$ can exist two quadrilateral disc types
whose surgery gives a pseudo-triangular disc type.

Now we consider the component $K_{i_{0}}$ which contains $P_{i_{0}}$ and we
distinguish two cases. In the first case $K_{i_{0}}$ is a normal surface and
there is nothing to do. In the second case, $K_{i_{0}}$ must have necessarily
a cyclic tunnel. If we denote by $i(\mathcal{D}^{(1)},K_{i_{0}})$ the number
of points that $K_{i_{0}}$ intersects the 1-skeleton $\mathcal{D}^{(1)},$ then
it is possible, by performing an isotopy in order to minimize $i(\mathcal{D}%
^{(1)},K_{i_{0}}),$ to remove the quadrilateral disc type $P_{i_{0}}$ and thus
$K_{i_{0}}$ can be reduced to a linking surface. Remark also that in this
second case $K_{i_{0}}$ cannot contain a pseudo-triangular disc type.

To deal with the second case, we consider a tetrahedron $D_{j_{1}}$ which
contains a tunnel disc type $T$ with $T\subset K_{i_{0}}.$ Let $t$ (resp.
$t^{\prime})$ be the arc in $\partial T$ whose end-points belong to the same
edge of $D_{j_{1}}$ and let us call \textit{tunnel face} of $D_{j_{1}}$ the
face $R_{j_{1}}$ (resp. $R_{j_{1}}^{\prime})$ of $D_{j_{1}}$ which contains
$t$ (resp. $t^{\prime}).$ Now, within $D_{j_{1}}$ we perform a surgery of
triangular type which does not produce a tunnel disc type; this surgery will
be referred to below as \textit{non-trivial surgery } in $D_{j_{1}}.$ Here we
need the following claim.

\textit{Claim 2}: The non-trivial surgery in $D_{j_{1}}$ can be extended to a
compatible surgery in all the tetrahedra $D_{j_{1}},...,D_{j_{k}}$ of
$\mathcal{D}$ whose union contains $K_{i_{0}},$ and also this non-trivial
surgery is compatible with the initial surgery in $D_{i_{0}}$, that is, it
respects the existence of $P_{i_{0}}.$

\textit{Proof of Claim 2}. For each $i=1,..k,$ the tunnel faces $R_{j_{i}},$
$R_{j_{i}}^{\prime}$ of $D_{j_{i}}$ are only glued with tunnel faces of
tetrahedra adjacent to $D_{j_{i}}$ along these faces. Hence, we perform
non-trivial surgeries in the tetrahedra $D_{j_{1}},...,D_{j_{k}}$ in order to
avoid tunnel components in them. The new disc types in each $D_{j_{i}}$ induce
in its faces which are different from $R_{j_{i}}$ and $R_{j_{i}}^{\prime}$
exactly the same arcs, as the surgery in $D_{j_{i}}$ which produces tunnel
disc type. Therefore the non-trivial surgery in each $D_{j_{i}}$ does not
affect the disc types in $D_{i_{0}}$ and thus it is compatible with the
initial surgery in $D_{i_{0}}.$

From Claim 2 we may assume that $K_{i_{0}}$ does not contain a cyclic tunnel
and we set $S_{0}=K_{i_{0}}.$ This surface $S_{0}$ is not a linking surface
since it contains a quadrilateral disc type. Finally, considering if necessary
a finite covering of $M_{d}$ we may assume that $S_{0}$ is orientable.
\end{proof}

The following proposition is rather technical but not difficult. Nevertheless,
we give a proof for the reader's convenience.

\begin{proposition}
\label{finite branched covering} There is a branched covering $M^{\prime}$ of
$M$ with branch locus the edges of $\mathcal{D}$ and a topological ideal
triangulation $\mathcal{D}^{\prime}$ of $M^{\prime}$ such that $i_{\mathcal{D}%
^{\prime}}(e^{\prime})\geq6$ for each edge $e^{\prime}$ of $\mathcal{D}%
^{\prime}.$
\end{proposition}

\begin{proof}
Let $e_{i},$ $i=1,..,n$ be the edges of $\mathcal{D}$. A \textit{cycle of
faces around }$e_{i}$ in $M$ consists of faces $F_{1}^{i},..,F_{k_{i}}^{i}$ of
$\mathcal{D}$ such that:

(i) For each $i,$ $e_{i}$ is a common edge of all $F_{1}^{i},..,F_{k_{i}}%
^{i}.$

(ii) For each $i,$ $F_{j}^{i}$ and $F_{j+1(\operatorname{mod}k_{i})}^{i}$
belong to the same ideal tetrahedron of $\mathcal{D}$.

We consider the triangulated ideal polyhedron $B$ of Lemma
\ref{special triangulation}.

Now, for every $e\in \mathcal{D}^{(1)}$ there exist pairs of boundary faces
$(F_{1}^{e},F_{1^{\prime}}^{e}),$ $(F_{2}^{e},F_{2^{\prime}}^{e}),...,$
$(F_{m_{e}}^{e},F_{m_{e}^{\prime}}^{e})$ of $B$ such that:

(1) $F_{i}^{e}$ and $F_{i^{\prime}}^{e}$ have a common edge, say $e_{i}^{B}%
\in \mathcal{B}^{(1)},$ $i=1,..,m_{e}$

(2) If we glue $F_{1^{\prime}}^{e}$ with $F_{2}^{e},$ $F_{2^{\prime}}^{e}$
with $F_{3}^{e},$ $...,$ $F_{m_{e}^{\prime}}^{e}$ with $F_{1}^{e},$ via gluing
homeomorphisms of $M,$ we obtain a manifold $M_{e}$ with boundary such that
all the edges $e_{i}^{B}$ are matched together to a common edge, say
$\overline{e},$ and the cycles of faces around $\overline{e}$ in $M_{e}$ and
around $e$ in $M$ are the same.

The positive integer $m_{e}$ will be referred to as the \textit{weight of }%
$e$\textit{ on the boundary of }$B.$

Considering $k$ copies of $B,$ say $B_{1},...,B_{k},$ we denote by
$(F_{j,1}^{e},F_{j,1^{\prime}}^{e}),$ $(F_{j,2}^{e},F_{j,2^{\prime}}%
^{e}),...,$ $(F_{j,m_{e}}^{e},F_{j,m_{e}^{\prime}}^{e})$ the boundary faces of
$B_{j}$ corresponding to the boundary faces $(F_{1}^{e},F_{1^{\prime}}^{e}),$
$(F_{2}^{e},F_{2^{\prime}}^{e}),...,$ $(F_{m_{e}}^{e},F_{m_{e}^{\prime}}^{e})$
of $B.$ Then, for each $j=1,..,k,$ we glue $F_{j,1^{\prime}}^{e}$ with
$F_{j,2}^{e},$ $F_{j,2^{\prime}}^{e}$ with $F_{j,3}^{e},$ $...,$
$F_{j,r^{\prime}}^{e}$ with $F_{j,r+1}^{e},$ $...,$ $F_{j,m_{e}^{\prime}}^{e}$
with $F_{(j+1)\operatorname{mod}k,1}^{e},$ via gluing homeomorphisms of $M.$
Thus, we obtain a manifold $M_{e}^{\prime}$ with boundary which is a $k$
branched cover of $M_{e}$ with branch locus the edge $\overline{e}.$ If we
denote by $\overline{e}^{\prime}$ the preimage of $\overline{e}$ in
$M_{e}^{\prime}$ we have that the index $i_{\mathcal{D}_{e}^{\prime}%
}(\overline{e}^{\prime})\geq km_{e},$ where $\mathcal{D}_{e}^{\prime}$ denotes
the ideal triangulation of $M_{e}^{\prime}$ and therefore $i_{\mathcal{D}%
_{e}^{\prime}}(\overline{e}^{\prime})\geq6$ for a suitable $k.$

Now, let $m_{0}=m_{e_{1}}\cdot m_{e_{2}}\cdot...\cdot m_{e_{n}},$ where
$m_{e_{i}}$ is a weight of $e_{i}$ on the boundary of $B$ and let
$k_{0}=k\cdot m_{0},$ where $k$ is a positive integer. Considering $k_{0}$
copies of $B,$ we may apply our previous method for the construction of
$M_{e_{i}}^{\prime},$ simultaneously for all $e_{i}.$ We take in this way a
triangulated ideal manifold $M^{\prime}$ without boundary, equipped with an
ideal triangulation $\mathcal{D}^{\prime},$ such that, $M^{\prime}$ is a
branched cover of $M$ with branched locus the edges $e_{i}$ of $\mathcal{D}.$
Obviously we may choose the integer $k$ above so that $i_{D^{\prime}%
}(e^{\prime})\geq6$ for each edge $e^{\prime}$ of $\mathcal{D}^{\prime}.$
\end{proof}

\noindent \textbf{Remark }Proposition \ref{finite branched covering} proves
just the existence of branched coverings and does not deal with the problem of
finding the "best" branched coverings i.e. branched coverings which have
minimal number of sheets around the edges or which have some other specific
properties. A similar remark is also valid for Proposition
\ref{finite covering}.

\section{Proof of the main theorem}

In order to prove Theorem \ref{main} we need two auxiliary lemmata.

Assuming that $\Delta_{0}$ is a regular, hyperbolic, ideal tetrahedron we have
the following:

\begin{lemma}
\label{special disc types}$(a)$ A triangular disc type in $\Delta_{0}$ can be
isotoped, via disc types, to an equilateral Euclidean triangle $P.$

$(b)$ A quadrilateral disc type in $\Delta_{0}$ can be isotoped, via disc
types, to a square $Q$ which is a geodesic surface in $\Delta_{0}$ and which
has all its angles equal to $\pi/3.$
\end{lemma}

\begin{proof}
$(a)$ A triangular disc type in $\Delta_{0}$ can be isotoped to a
horospherical section $P$ in the neighborhood of an ideal vertex of
$\Delta_{0}.$ Obviously, $P$ is an equilateral Euclidean triangle.

$(b)$ We consider the regular, hyperbolic, ideal tetrahedron $\Delta_{0}$ in
the hyperbolic ball model $\mathbb{B}^{3}=\{(x,y,z)\in$ $\mathbb{R}^{3}%
:x^{2}+y^{2}+z^{2}<1\}.$ Considering the geodesic plane $\Pi$ passing trough
the origin $O$ and perpendicular to an edge $e$ of $\Delta_{0}$ the
intersection $\Delta_{0}\cap \Pi$ defines a disc type $Q$. The symmetry of
$\mathbb{B}^{3}$ permits to prove that $Q$ intersects perpendicularly each
edge $e^{\prime}$ of $\Delta_{0}$ with $e^{\prime}\cap Q\neq \varnothing.$
Therefore all the angles of $Q$ are equal to $\pi/3.$ Also, it is easily
verified that all the sides of $Q$ are equal, so $Q$ is a hyperbolic square.
\end{proof}

Now, let $T$ be an ideal, hyperbolic triangle in $\mathbb{H}^{2}$ and let $p,$
$q,$ $r$ be the points in the sides of $T$ where the inscribed circle $C$ in
$T$ intersects the sides. The following lemma can be proven easily, by
elementary computations.

\begin{lemma}
\label{computational lemma}If $h_{0}$ is the length of the greatest
horospherical arc in $T$ centered at an ideal vertex of $T,$ and if $l_{0}$ is
the distance between $p,q$ then $l_{0}<h_{0}.$
\end{lemma}

\begin{proof}
In the hyperbolic half-plane we consider the ideal triangle $T$ with vertices
the points $(-1,0),$ $(1,0)$ and $\infty.$ Obviously $h_{0}=2$, while
$l_{0}<\log3.$ This proves the lemma.
\end{proof}

Now, we are able to prove the following theorem.

\begin{theorem}
\label{incompressible surface}Let $M$ be a triangulated ideal manifold
equipped with a topological ideal triangulation $\mathcal{D}.$ We assume that
$i_{\mathcal{D}}(e)\geq6$ for each edge $e$ of $\mathcal{D}$ and that $M$
contains a closed normal surface $S$ which is not a linking surface. Then $S$
is incompressible.
\end{theorem}

\begin{proof}
The manifold $M$ equipped with the regular, ideal structure $d$ defined in
Section 2, becomes a regular, negatively curved cusped manifold and is denoted
by $M_{d}.$ Thus, each tetrahedron of the triangulation $\mathcal{D}$ of
$M_{d}$ is a regular, hyperbolic ideal tetrahedron.

\begin{figure}[ptb]
\begin{center}
\includegraphics[scale=0.9]{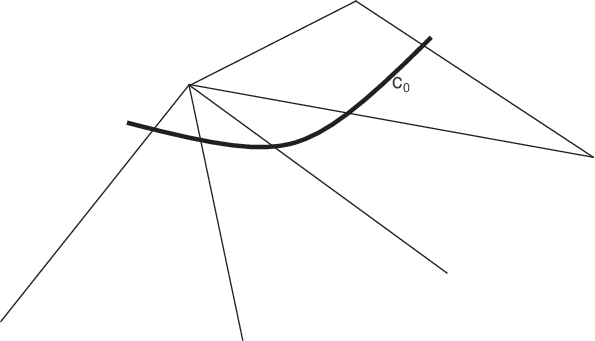}
\end{center}
\caption{A local image of a geodesic in the neighborhood of a vertex of $S$.}%
\end{figure}

The closed normal surface $S$ of $M_{d}$ can be isotoped to a surface, always
denoted by $S,$ which is of non-positive curvature. In fact, from Lemma
\ref{special disc types} each quadrilateral disc type can be isotoped to a
hyperbolic square with all its angles equal to $\pi/3$ and every triangular
disc type can be isotoped to an equilateral Euclidean triangle. Furthermore,
from Lemma \ref{computational lemma} all these geometric disc types can be
chosen to have sides of equal length and therefore they can be matched
together. Thus, the surface $S$ is equipped with a cell decomposition
$\mathcal{T}$ induced on $S$ from the ideal triangulation $\mathcal{D}$ of
$M_{d}.$ The angle around each vertex of $S$ is greater than $2\pi$ since
$i_{\mathcal{D}}(e)\geq6$ for each edge $e$ of $\mathcal{D}.$ On the other
hand, each geometric disc type is either Euclidean or hyperbolic. Therefore,
we deduce that $S$ is of non-positive curvature i.e. $S$ satisfies locally the
$CAT(0)$ inequality (see for example Theorem 3.13 in \cite{Paulin}).

Let $G\subset S$ be the graph formed by all the edges of $\mathcal{T}.$ Each
$v\in \mathcal{T}^{(0)}$ is a vertex of $G$ and let us denote by
$i_{\mathcal{T}}(v)$ the index of $v$ in the graph $G.$ Obviously we have
$i_{\mathcal{T}}(v)\geq6.$ In order to prove that $S$ is incompressible we
will use in a meaningful way that $S$ has non-positive curvature and that
$M_{d}$ has negative curvature.

Let $a$ be a simple, closed, essential curve in $S;$ we will show that $a$ is
non-contractible in $M_{d}.$ First we remark that we may replace $a$ by a
closed geodesic $c_{0}$ of $S$ which is freely homotopic to $a.$ Indeed, it is
well known that $S=\widetilde{S}/\Gamma,$ where $\widetilde{S}$ is the
universal covering of $S$ and $\Gamma$ a discrete group of isometries of
$\widetilde{S}$ acting freely on $\widetilde{S}.$ Furthermore, $\Gamma$ is
isomorphic to the fundamental group $\pi_{1}(S).$ Since $\widetilde{S}$ is a
$CAT(0)$ space, the curve $a$ defines an element $\phi$ of $\Gamma$ which is a
hyperbolic isometry of $\widetilde{S}.$ Therefore the geodesic of
$\widetilde{S}$ joining the points $\phi(-\infty)$ and $\phi(\infty)$ projects
to a closed geodesic $c_{0}$ of $S$ and it is easy to see that $c_{0}$ is
freely homotopic to $a$ (see for example Lemma 8 of \cite{CPT}).

\begin{figure}[ptb]
\begin{center}
\includegraphics[scale=0.7]
{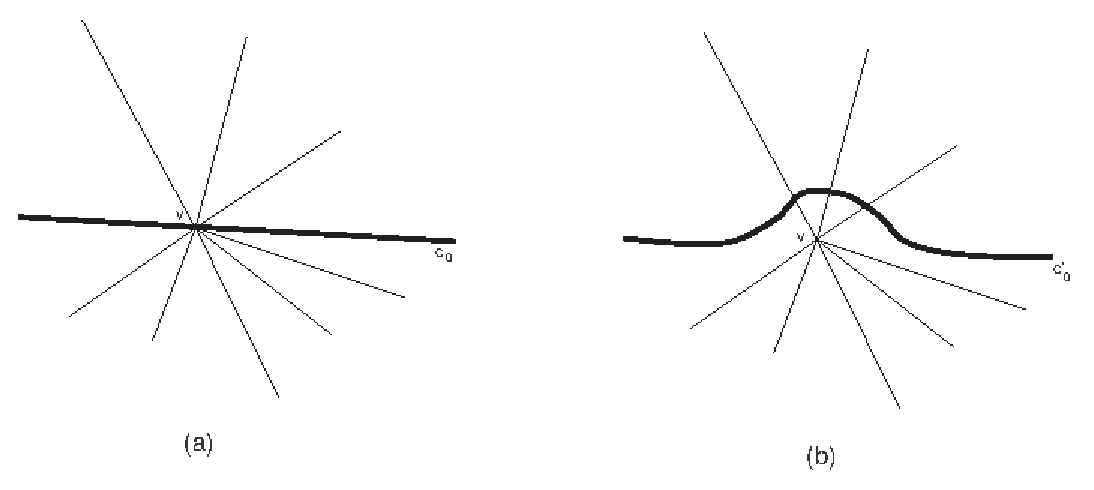}
\end{center}
\caption{A local modification of a geodesic in the neighborhood of a vertex.}%
\end{figure}

\textit{Claim 1}: If $c_{0}$ does not contain any vertex of $\mathcal{T}$ then
$c_{0}$ cannot intersect more than three consecutive edges of $\mathcal{T}$
abutting to the same vertex, see Figure 4.

Indeed, if Claim 1 is not true then a geodesic triangle would be formed in $S$
whose sum of angles would be strictly greater than $\pi.$ But this is
impossible since $S$ has non-positive curvature.

We assume now that $c_{0}$ contains some vertices of $\mathcal{T}$ and we
distinguish the following cases:

\begin{itemize}
\item (1) $c_{0}\subset \mathcal{T}^{(1)}$ i.e. $c_{0}$ consists of edges of
$\mathcal{T}.$

\item (2) There are edges $e,$ $e^{\prime}$ of $\mathcal{T}$ having a common
vertex $v$ such that $c_{0}$ contains $e$ but $c_{0}$ does not contain
$e^{\prime}.$

\item (3) $c_{0}$ passes through a vertex $v\in \mathcal{T}^{(0)}$ and $c_{0}$
does not contain any edge $e$ of $\mathcal{T}$ containing $v$ as a vertex.
\end{itemize}

\textit{Claim 2}: Let $v$ be a vertex of $\mathcal{T}$ with $v\in c_{0}$ and
let $c_{0}(t),$ $t\in \lbrack0,1]$ be a parametrization of $c_{0}$ with
$c_{0}(t_{0})=v.$ Then, we may find an open neighborhood $U$ of $v$ in $S$ and
an $\varepsilon>0$ such that $c_{0}((t_{0}-\varepsilon,t_{0}+\varepsilon))$
separates \ $U$ in two half-discs, say $U^{+},$ $U^{-}$ such that:

\begin{itemize}
\item In the case (1) there are at least two edges in the interior of $U^{+}$
(respectively in $U^{-})$ emanating from $v.$

\item In the cases (2) and (3) there are at least three edges in the interior
of $U^{+}$ (respectively in $U^{-})$ emanating from $v.$
\end{itemize}

Indeed, Claim 2 follows easily from the fact that $c_{0}$ is a geodesic of $S$
and thus the angle formed at $v$ in $S$ by the geodesic segments $[c_{0}%
(t_{0}-\varepsilon),c_{0}(t_{0})]$ and $[c_{0}(t_{0}),c_{0}(t_{0}%
+\varepsilon)]$ must be $\geq \pi$ at each side of $c_{0}.\smallskip$

Furthermore, in case (3) the following particular situation can happen: in
$U^{+}$ (or in $U^{-}$) there are exactly three edges emanating from $v,$ see
Figure 5(a). In this case the vertex $v$ will be referred to as a
\textit{specific vertex }of $c_{0}$ and we claim that:

\textit{Claim 3}: In case (3), we may isotope locally the geodesic $c_{0}$ in
$U$ and take a curve $c_{0}$ which traverses $U$ through $U^{+}$ and
intersects the three edges lying in $U^{+}$ in interior points; furthermore
$c_{0}^{\prime}$ does not intersect more than three consecutive edges of
$\mathcal{T}$ abutting at $v;$ see Figure 5(b).\smallskip

In case (1), $c_{0}\subset \mathcal{D}^{(2)}$ and $c_{0}$ in not homotopic to a
cusp of the 2-dimensional ideal polyhedron $\mathcal{D}^{(2)}$ (for the
definition of cusps in 2-dimensional ideal polyhedra, see Definition 1.8 in
\cite{ChTs}). Therefore $c_{0}$ can be isotoped in $\mathcal{D}^{(2)}$ to a
geodesic $\gamma_{0}$ of $\mathcal{D}^{(2)}.$ Let $e$ be an edge of
$\mathcal{D}$ and let $p\in \gamma_{0}\cap e.$ Then, there are at least six
elements of $\mathcal{D}$ i.e. regular hyperbolic tetrahedra of $\mathcal{D},$
which have $e$ as a common side and furthermore, locally in a neighborhood of
$p$ in $M,$ there are at least three elements of $\mathcal{D}$ at each side of
$\gamma_{0}.$ This implies that $\gamma_{0}$ is a geodesic of $M$ and hence
$\gamma_{0}$ is non-contractible in $M.$

In the following we will also show that $c_{0}$ is non-contractible in $M$ in
cases (2) and (3). Whether $c_{0}$ contains specific vertices the curve
$c_{0}$ will be used instead of $c_{0}.$ Our goal is to construct a surface
$T_{0}$ with the following properties:

$(i)$ $T_{0}$ is homeomorphic to an annulus $S^{1}\times \lbrack0,1]$ and
consists of geometric disc types of the triangulation $\mathcal{T}.$

$(ii)$ If we denote by $\mathcal{T}_{0}$ the triangulation of $T_{0},$ then
all the vertices of $\mathcal{T}_{0}$ belong to $\partial T_{0}$.

$(iii)$ If $v$ is a vertex of $\mathcal{T}_{0}$ then the number of edges of
$\mathcal{T}_{0}$ in the interior of $T_{0}$ abutting at $v$ is less than or
equal to $i_{\mathcal{T}}(v)-4.$

In fact, such an annulus $T_{0}$ with the above properties can be built by
gluing successively all disc types of $\mathcal{T}$ that $c_{0}$ transverses.
Observe here that a disc type can be used several times in the construction of
$T_{0}.$ Obviously $T_{0}$ satisfies properties $(i)$ and $(ii)$ above.
Property $(iii)$ follows from Claims (1) - (3).\smallskip

The following claim describes the relation of $T_{0}$ with $S.$

\textit{Claim 4}: There exists a map $f:T_{0}\rightarrow S$ such that:

$(i^{\prime})$ If $\Delta$ is a 2-cell of $\mathcal{T}_{0}$ then $f_{|\Delta
}:\Delta \rightarrow S$ is an isometry.

$(ii^{\prime})$ $f$ is a local isometry.

\textit{Proof of Claim 4}.

From the construction of $T_{0}$ statement $(i^{\prime})$ is clear. To prove
statement $(ii^{\prime})$ we remark first that $T_{0}$ is a geodesic space.
Now, let $U\subset T_{0}$ be an open convex neighborhood such that $f_{|U}$ is
an embedding. Let $p,q\in U$ and let $[p,q]$ be the unique geodesic segment
connecting $p$ and $q$ in $U.$ If $[p,q]$ belongs to the interior of $T_{0}$
then $[p,q]$ is obviously a geodesic segment of $S.$ If $[p,q]\cap \partial
T_{0}\neq \varnothing,$ then using the property $(iii)$ of $T_{0}$ above, we
will show that $[p,q]$ is also a geodesic segment in $S.$ In fact, if $w$ is a
vertex of $\mathcal{T}_{0}$ belonging to $[p,q]$ then the geodesic segments
$[w,p]$ and $[w,q]$ form at $w$ an angle $\geq \pi$ from each side of $[p,q]$
in $S.$ Therefore $[p,q]$ is also a geodesic segment in $S$ and thus the
distance $d_{T_{0}}(p,q)$ of $p,q$ in $T_{0}$ is equal to the distance
$d_{S}(p,q)$ of $p,q$ in $S.$ This implies that $f$ is a local isometry and
Claim 4 is proved.

The construction of $T_{0}$ permits to construct a 3-dimensional manifold
$N_{0}$ containing $c_{0}$ such that:

$(1^{\prime})$ $N_{0}$ consists of regular, hyperbolic ideal tetrahedra of
$\mathcal{D}$ and $T_{0}\subset N_{0}.$

$(2^{\prime})$ $N_{0}$ is homeomorphic to the solid torus $R=[0,1]\times
\lbrack0,1]\times S^{1}$ minus finitely many points removed from $\partial R.$

$(3^{\prime})$ Each edge $e$ of $\mathcal{D}_{0}$ is lying in $\partial
N_{0}.$

$(4^{\prime})$ For each edge $e$ of $\mathcal{D}_{0}$ there is an open
neighborhood $U_{N_{0}}(e)$ of $e$ in $N_{0}$ which is isometrically embedded
in $M.$

The construction of $N_{0}$ with the previous features follows from the
construction of $T_{0}.$ In fact, since every 2-cell of $\mathcal{T}_{0}$ is a
disc type in some hyperbolic ideal tetrahedron of $\mathcal{D}$ the
construction of $T_{0}$ leads naturally to the construction of $N_{0}$ having
properties $(1^{\prime})$-$(4^{\prime}).$ More precisely, two disc types of
$T_{0}$ which share a common edge, say $d,$ determine the two hyperbolic ideal
tetrahedra of $\mathcal{D}$ which are glued along a common face containing
$d.$ Furthermore, $N_{0}$ has negative curvature. In fact, this follows from
Proposition 11.6 of \cite{Bridson-Haefliger}, since $N_{0}$ is constructed by
gluing successively hyperbolic ideal tetrahedra along convex subsets. 

Let us denote by $\mathcal{D}_{0}$ the ideal triangulation of $N_{0}$ and let
$\phi(e,N_{0})$ be the angle around $e$ in $N_{0},$ which is defined as the
sum of all dihedral angles at $e$ of tetrahedra of $\mathcal{D}_{0}$ which
share $e$ as a common edge. Property $(4^{\prime})$ above allows us to define
the complement of this angle $\phi^{c}(e,N_{0})$ in $M_{d}.$ Then, Property
$(iii)$ of the surface $T_{0}$ implies that $\phi^{c}(e,N_{0})\geq \pi.$
Therefore, exactly as in Claim 4, we may prove that there exists a local
isometry $h:N_{0}\rightarrow M_{d}$.

Now, if $c_{0}$ is a cuspidal curve in $N_{0}$ it follows that $c_{0}$ will be
also a cuspidal curve in $M$. Therefore, from Lemma \ref{horospheres} $c_{0}$
is non-contractible in $M.$ If $c_{0}$ is not cuspidal then in the free
homotopy class of $c_{0}$ there is a closed geodesic $\gamma_{0}\subset$
$N_{0}.$ In fact, the lift of $c_{0}$ in the universal covering $\widetilde
{N_{0}}$ of $N_{0}$ defines two different points $\xi,\eta$ in the boundary
$\partial \widetilde{N_{0}},$ Then the geodesic of $\widetilde{N_{0}}$ which
joins $\xi$ and $\eta$ projects to $\gamma_{0}.$ Finally, since $\phi
^{c}(e,N_{0})\geq \pi$ for each edge $e$ of $\mathcal{D}_{0},$ using the same
arguments as in Claim 4, we deduce that $\gamma_{0}$ is also a geodesic of
$M_{d}.$ Therefore $c_{0}$ is not contractible in $M_{d}$ and thus $S$ is incompressible.
\end{proof}

Now we are able to prove our main theorem restated below.

\begin{theorem}
Let $M$ be an orientable 3-manifold triangulated by finitely many topological
ideal tetrahedra. If $M$ admits a regular, negatively curved, ideal
structure\textit{, }there exists a finite covering space $\widetilde{M}$ of
$M$ containing an essential closed surface.
\end{theorem}

\begin{proof}
From Proposition \ref{finite covering} and Theorem
\ref{existence normal surface}, there exists a finite covering space
$\widetilde{M}$ of $M$ equipped with a topological ideal triangulation such
that $\widetilde{M}$ contains a closed normal surface $S$ which is not a
linking surface. The covering $\widetilde{M}$ admits a regular, ideal
structure $d$ and becomes a regular, negatively curved cusped manifold having
an embedded normal surface $S.$ Hence, from Theorem
\ref{incompressible surface}, $S$ is essential.
\end{proof}

\end{document}